\documentclass[12pt]{amsart}
\usepackage{amsmath,amsthm,amsfonts,amscd,amssymb,eucal,latexsym,mathrsfs}
\newtheorem{theorem}{Theorem}

\newtheorem{lemma}[theorem]{Lemma}

\theoremstyle{definition}

\newtheorem{conjecture}[theorem]{Conjecture}

\newcommand{\cL}{{\mathcal L}}

\newcommand{\Cb}{{\mathbb C}}

\newcommand{\Zb}{{\mathbb Z}}
\newcommand{\Tb}{{\mathbb T}}

\newcommand{\Rb}{{\mathbb R}}

\newcommand{\tr}{{\rm tr}}

\allowdisplaybreaks

\begin{document}

\title[Appendix]{Appendix to V. Mathai and J. Rosenberg's paper  ``A noncommutative sigma-model"}
\author{Hanfeng Li}
\thanks{Partially supported by NSF Grant DMS-0701414.}
\address{Department of Mathematics \\
SUNY at Buffalo \\
Buffalo, NY 14260-2900, U.S.A.} \email{hfli@math.buffalo.edu}
\date{December 1, 2009}


\maketitle

This short note is an appendix to \cite{MR}.

Let $\theta\in \Rb$. Denote by $A_{\theta}$ the rotation $C^*$-algebra generated by unitaries $U$ and $V$ subject
to $UV=e^{2\pi i \theta}VU$, and by $A_{\theta}^{\infty}$ its canonical smooth subalgebra.
Denote by $\tr$ the canonical faithful tracial state on $A_{\theta}$ determined by $\tr(U^mV^n)=\delta_{m, 0}\delta_{n, 0}$ for
all $m, n\in \Zb$.
%Denote by $H$ the Hilbert space associated to  the GNS representation of $A_{\theta}$ for $\tr$, and denote
% by $\|\cdot \|_2$ its norm. We shall identify $A_{\theta}$ as a subspace of $H$ canonically.
Denote by $\delta_1$ and $\delta_2$ the unbounded closed $*$-derivations of $A_{\theta}$ defined on some dense subalgebras of $A_{\theta}$ and
determined by
$\delta_1(U)=2\pi i U$, $\delta_1(V)=0$, and $\delta_2(U)=0$, $\delta_2(V)=2\pi i V$. The {\it energy} \cite{Rosenberg}, $E(u)$, of a unitary $u$ in $A_{\theta}$
is defined as
\begin{eqnarray} \label{energy:eq}
 E(u)=\frac{1}{2}\tr(\delta_1(u)^*\delta_1(u)+\delta_2(u)^*\delta_2(u))
\end{eqnarray}
%$$ E(u)=\frac{1}{2}(\|\delta_1(u)\|_2^2+\|\delta_2(u)\|_2^2)$$
when $u$ belongs to the domains of $\delta_1$ and $\delta_2$, and $\infty$ otherwise.
%A unitary $u$ in $A_{\theta}^{\infty}$ is called {\it harmonic} if it is a critical point of the functional $E$.

Rosenberg has the following conjecture \cite[Conjecture 5.4]{Rosenberg}.
%(I am not sure whether Rosenberg is assuming $\theta$ to be irrational or not):

\begin{conjecture} \label{unitary:conj}
For any $m, n\in \Zb$, in the connected component of $U^mV^n$ in the unitary group of $A_{\theta}^{\infty}$, the functional $E$ takes its minimal value exactly at the scalar multiples
of $U^mV^n$.
% for some $m, n\in \Zb$.
%These are the only harmonic unitaries in this component.
\end{conjecture}

For a $*$-endomorphism $\varphi$ of $A_{\theta}^{\infty}$, its {\it energy} \cite{MR},  $\cL(\varphi)$, is defined as $2E(\varphi(U))+2E(\varphi(V))$.
Mathai and Rosenberg's Conjecture 3.1 in  \cite{MR} about the minimal value of $\cL(\varphi)$
follows directly from Conjecture~\ref{unitary:conj}.

%(I am not sure whether Mathai and Rosenberg are assuming $\varphi$ to
%be unital or not):
%
%\begin{conjecture} \label{map:conj}
%Let
%\begin{eqnarray*}
%\cA=\left(
%\begin{matrix}
%p & q \\
%r & s
%\end{matrix}
%\right) \in \SL(2, \Zb).
%\end{eqnarray*}
%Among the $*$-endomorphisms $\varphi$ of $A_{\theta}^{\infty}$ inducing $\cA$ on $K_1(A_{\theta})\cong \Zb^2$ , $\cL(\varphi)$ gets minimal values
%at the $\varphi's$ sending $U$ and $V$ to scalar multiples of $U^pV^r$ and $U^qV^s$ respectively.
%\end{conjecture}

%When $\theta$ is irrational, the natural map $\pi_0(U(A_{\theta}))\rightarrow K_1(A_{\theta})$ is a bijection \cite[Theorem 3.4]{Rieffel87}.
%Since $A_{\theta}^{\infty}$ is a smooth algebra in $A_{\theta}$, the natural map $\pi_0(U(A_{\theta}^{\infty}))\rightarrow \pi_0(U(A_{\theta}))$ is a %bijection.

Denote by $H$ the Hilbert space associated to  the GNS representation of $A_{\theta}$ for $\tr$, and denote
 by $\|\cdot \|_2$ its norm. We shall identify $A_{\theta}$ as a subspace of $H$ as usual.
Then (\ref{energy:eq}) can be rewritten as
$$ E(u)=\frac{1}{2}(\|\delta_1(u)\|_2^2+\|\delta_2(u)\|_2^2).$$

Now we prove Conjecture~\ref{unitary:conj}, and hence also prove Conjecture 3.1 of \cite{MR}.
%The following theorem proves the first part of Conjecture~\ref{unitary:conj}, and implies the unital $*$-endomorphism case of %Conjecture~\ref{map:conj}.

\begin{theorem} \label{min:thm}
Let $\theta \in \Rb$ and $m, n\in \Zb$. Let $u\in A_{\theta}$ be a unitary whose class in $K_1(A_{\theta})$ is the same as that of $U^mV^n$.
Then $E(u)\ge E(U^mV^n)$, and ``='' holds if and only if $u$ is a scalar multiple of $U^mV^n$.
\end{theorem}
\begin{proof}
We may assume that $u$ belongs to the domains of $\delta_1$ and $\delta_2$.
Set $a_j=u^*\delta_j(u)$ for $j=1, 2$. For any closed $*$-derivation $\delta$ defined on a dense subset of a unital $C^*$-algebra $A$
and any tracial state $\tau$ of $A$ vanishing on the range of $\delta$,
if unitaries $v_1$ and $v_2$ in the domain of $\delta$ have the same class in $K_1(A)$, then $\tau(v^*_1\delta(v_1))=\tau(v^*_2\delta(v_2))$
\cite[page 281]{PW}.  Thus
\begin{eqnarray*}
\tr(a_j)=\tr((U^mV^n)^*\delta_j(U^mV^n))=
\left\{
\begin{array}{lr}
2\pi i m & \mbox{ if } j=1;\\
2\pi i n & \mbox{ if } j=2.
\end{array}
\right.
\end{eqnarray*}
We have
\begin{eqnarray*}
\|\delta_j(u)\|_2^2&=&\|a_j\|_2^2=\|\tr(a_j)\|_2^2+\|a_j-\tr(a_j)\|_2^2\\
&\ge &\|\tr(a_j)\|_2^2=|\tr(a_j)|^2\\
&=&\left\{
\begin{array}{lr}
4\pi^2 m^2 & \mbox{ if } j=1;\\
4\pi^2 n^2 & \mbox{ if } j=2,
\end{array}
\right.
\end{eqnarray*}
and
``='' holds if and only if $a_j=\tr(a_j)$. It follows
that $E(u)\ge 2\pi^2(m^2+n^2)$, and ``='' holds if and only if $\delta_1(u)=2\pi im u$ and $\delta_2(u)=2\pi in u$.
Now the theorem follows from the fact that the elements $a$ in $A_{\theta}$ satisfying $\delta_1(a)=2\pi i ma$ and $\delta_2(a)=2\pi i na$ are exactly the
scalar multiples of $U^mV^n$.
\end{proof}

When $\theta\in \Rb$ is irrational, the $C^*$-algebra $A_{\theta}$ is simple \cite[Theorem 3.7]{Slawny},
has real rank zero \cite[Theorem 1.5]{BKR}, and is
an $A\Tb$-algebra \cite[Theorem 4]{EE}. It is a result of Elliott that for any pair of
$A\Tb$-algebras with real rank zero, every homomorphism between their graded $K$-groups preserving the
graded dimension range is induced by a $*$-homomorphism between them \cite[Theorem 7.3]{Elliott}.
The graded dimension range of
a unital simple $A\Tb$-algebra $A$ is the subset
$\{(g_0, g_1)\in K_0(A)\oplus K_1(A): 0 \lvertneqq g_0\le [1_A]_0\}\cup (0, 0)$ of the graded $K$-group
$K_0(A)\oplus K_1(A)$ \cite[page 51]{Rordam}. It
follows that, when $\theta$ is irrational,  for any group endomorphism $\psi$ of $K_1(A_{\theta})$,
there is a unital $*$-endomorphism $\varphi$ of $A_{\theta}$ inducing $\psi$ on $K_1(A_{\theta})$.
%It was shown in \cite{BCEN, CEGJ} that if $\theta$ is irrational and $\varphi$ restricts to a $*$-automorphism of $A_{\theta}^{\infty}$,
%then $\psi$ must be an automorphism of the rank-two free abelian group $K_1(A_{\theta})$
%with determinant $1$.
It is an open question when one can choose
$\varphi$ to be smooth in the sense of preserving $A_{\theta}^{\infty}$, though
it was shown in \cite{BCEN, CEGJ} that if $\theta$ is irrational and $\varphi$ restricts to a $*$-automorphism of $A_{\theta}^{\infty}$,
then $\psi$ must be an automorphism of the rank-two free abelian group $K_1(A_{\theta})$
with determinant $1$.
  When $\psi$ is the zero endomorphism, from Theorem~\ref{min:thm}
one might guess that $\cL(\varphi)$ could be arbitrarily small. It is somehow surprising, as we show now, that in fact there is a common  positive lower bound for
$\cL(\varphi)$ for all $0<\theta<1$. This answers a question Rosenberg raised at the Noncommutative Geometry workshop at Oberwolfach in September 2009.

\begin{theorem} \label{lower bound:thm}
Suppose that $0<\theta<1$. For any unital $*$-endomorphism $\varphi$ of $A_{\theta}$, one has $\cL(\varphi)\ge 4(3-\sqrt{5})\pi^2$.
\end{theorem}

Theorem~\ref{lower bound:thm} is a  direct consequence of  the following lemma.

\begin{lemma} \label{lower bound:lemma}
Let $\theta\in \Rb$ and let $u, v$ be unitaries in $A_{\theta}$ with $uv=\lambda vu$ for some $\lambda \in \Cb\setminus \{1\}$.
Then $E(u)+E(v)\ge 2(3-\sqrt{5})\pi^2$.
\end{lemma}
\begin{proof}
We have
$$ \tr(uv)=\tr(\lambda vu)=\lambda\tr(uv),$$
and hence $\tr(uv)=0$. Thus
\begin{eqnarray*}
 -\tr(u)\tr(v)&=&\tr(uv-\tr(u)\tr(v))=\tr((u-\tr(u))v)+\tr(\tr(u)(v-\tr(v)))\\
 &=&\tr((u-\tr(u))v).
\end{eqnarray*}

We may assume that both $u$ and $v$ belong to the domains of $\delta_1$ and $\delta_2$. For any $m, n\in \Zb$, denote by $a_{m, n}$ the Fourier coefficient
$\left<u, U^mV^n\right>$ of $u$.
Then $a_{0, 0}=\tr(u)$, and
\begin{eqnarray*}
 (2\pi)^2\|u-\tr(u)\|_2^2&=&\sum_{m,n\in \Zb, m^2+n^2>0}|2\pi a_{m, n}|^2\\
 &\le &\sum_{m,n\in \Zb, m^2+n^2>0}|2\pi a_{m, n}|^2(m^2+n^2)\\
&=&\|\delta_1(u)\|_2^2+\|\delta_2(u)\|_2^2=2E(u).
\end{eqnarray*}
Thus
$$ |\tr(u)|^2=\|\tr(u)\|_2^2=\|u\|_2^2-\|u-\tr(u)\|_2^2\ge 1-\frac{1}{2\pi^2}E(u),$$
and
$$ |\tr((u-\tr(u))v)|\le \|(u-\tr(u))v\|_2=\|u-\tr(u)\|_2\le (\frac{1}{2\pi^2}E(u))^{1/2}.$$
Similarly, $|\tr(v)|^2\ge 1-\frac{1}{2\pi^2}E(v)$.
%We also have
%\begin{eqnarray*}
% -\tr(u)\tr(v)&=&\tr(uv-\tr(u)\tr(v))=\tr((u-\tr(u))v)+\tr(\tr(u)(v-\tr(v)))\\
% &=&\tr((u-\tr(u))v).
%\end{eqnarray*}

Write $\frac{1}{2\pi^2}E(u)$ and $\frac{1}{2\pi^2}E(v)$ as $t$ and $s$ respectively. We just need to show
that $t+s\ge 3-\sqrt{5}$. If $t\ge 1$ or $s\ge 1$, then this is trivial. Thus we may assume that $1-t, 1-s> 0$.
%$$ (1-\frac{1}{2\pi}(2E(u))^{1/2})(1-\frac{1}{2\pi}(2E(v))^{1/2})\le |\tr(u)\tr(v)|\le \frac{1}{2\pi}(2E(u))^{1/2}.$$
%Write $\frac{1}{2\pi}(2E(u))^{1/2}$ and $\frac{1}{2\pi}(2E(v))^{1/2}$ as $t$ and $s$ respectively.
%Then the above inequality becomes $(1-t)(1-s)\le t$.
Then
$$(1-t)(1-s)\le |\tr(u)\tr(v)|^2\le t.$$
Equivalently, $t(1-s)\ge 1-(t+s)$.
%Similarly, $2s+t\ge 1+ts$.
Without of loss generality, we may assume $s\ge t$.
Write $t+s$ as $w$. Then
$$ t(1-w/2)\ge t(1-s)\ge 1-(t+s)=1-w,$$
and hence
$$ w=t+s\ge \frac{1-w}{1-w/2}+\frac{w}{2}.$$
It follows that $w^2-6w+4\le 0$. Thus $w\ge 3-\sqrt{5}$.
%Adding these two inequalities together, we get $3(t+s)\ge 2+2ts\ge 2$.
%This proves the claim.
%
%Now we have
%$$ E(u)+E(v)=2\pi^2(t^2+s^2)\ge \pi^2(s+t)^2\ge \frac{4}{9}\pi^2.$$
\end{proof}

%%%%%%%%%%%%%%%%%%%%%%%%%%%%%%%%%%%%%%%%%%%%%%%%%%%%%%%%%%%%%%%%%%%%%%%%%%%%%%

\end{document}